\documentclass[12pt,draft,amsmath,amscd,amsbsy,amssymb]{amsart}
\usepackage{amssymb,latexsym, amsmath, amscd, array, graphicx}

\swapnumbers \numberwithin{equation}{section}

\theoremstyle{plain}

\newtheorem{thm}{Theorem}[section]

\newtheorem{lem}[thm]{Lemma}

\newtheorem{prop}[thm]{Proposition}
\newtheorem{ass}[thm]{Assertion}

\theoremstyle{definition}
\newtheorem{defin}[thm]{Definition}

\newtheorem{rem}[thm]{Remark}

\newtheorem{ex}[thm]{Example}

\newtheorem{question}[thm]{Question}

\DeclareMathOperator{\as}{{\rm asdim}}

\def\scr{\mathcal}

\def\Z{{\mathbb Z}}

\def\N{{\mathbb N}}

\def\1{\hbox{\rm\rlap {1}\hskip.03in{\rom I}}}
\def\Bbbone{{\rm1\mathchoice{\kern-0.25em}{\kern-0.25em}
{\kern-0.2em}{\kern-0.2em}I}}

\long\def\forget#1\forgotten{} %

\newcommand\ver[1]{\marginpar{\tiny Changed in Ver \VER}}

\date{\today}
\begin{document}

\title[On property C]{Asymptotic dimension, decomposition complexity, and Haver's property C }

\author{Alexander  Dranishnikov$^1$}\footnote{ Supported by NSF grant  DMS-0904278}
\author[M.~Zarichnyi]{Michael Zarichnyi}%

\begin{abstract} The notion of  the decomposition complexity was introduced in~\cite{GTY}
using a game theoretical approach. We introduce a notion of straight  decomposition complexity and compare it with
the original as well with the asymptotic property C. Then we define a game theoretical analog of Haver's property
 C in the classical dimension theory and compare it with the original.
\end{abstract}

\address{Alexander N. Dranishnikov, Department of Mathematics, University
of Florida, 358 Little Hall, Gainesville, FL 32611-8105, USA}
\email{dranish@math.ufl.edu}

\address{Michael Zarichnyi, Department of Mechanics and Mathematics, Lviv University, 1 Universytetska Str., 79000 Lviv, Ukraine; Faculty of Mathematics and Natural Sciences,
University of Rzesz\'ow, Rejtana 16 A, 35-310 Rze\-sz\'ow, Poland}
\email{mzar@litech.lviv.ua}



\maketitle

\section{Introduction}
Asymptotic dimension was introduced by Gromov to study finitely generated groups though the definition can
be applied for all metric spaces~\cite{Gr1}. It received a great deal of attention when Goulyang Yu proved
the Novikov higher signature conjecture for groups with finite asymptotic dimension~\cite{Yu1}. There were many other similar
results about groups and manifolds under assumption of finiteness of the asymptotic dimension or the asymptotic dimension of the
fundamental group~\cite{Dr2},\cite{Ba},\cite{CG},\cite{DFW}. When G. Yu introduced property A and proved the coarse Baum-Connes
for groups with property A, it was a natural problem to check whether every finitely presented group has this property. A construction
of a finitely presented group without property A was suggested by Gromov~\cite{Gr2} (see for detailed presentation~\cite{AD}). Gromov's random group construction
is an existence theorem. Still, it is a good question whether a given group (or class of groups) has property A. The answer is unknown for the
Thompson group $F$.

It turns out that the dimension theoretical approach to verification of property A proved to be quite productive. There are different extensions
of features of finite asymptotic dimension to asymptotically infinite dimensional spaces. Some of them came from the analogy with classical
dimension theory, some from function growth, and some from game theory. The asymptotic property C was defined in~\cite{Dr1} by analogy
with Haver's property C in the classical dimension theory. It was shown that the asymptotic property C implies the property A. Different versions
of asymptotic dimension growth were suggested in ~\cite{Dr3},\cite{CFY},\cite{DS1},\cite{DS2}. The property A was proven for groups with
the sublinear dimension growth~\cite{Dr1},\cite{CFY}, then for the polynomial dimension growth~\cite{Dr3}, and finally, for the subexponential dimension growth
~\cite{Oz}. Since dimension growth of a finitely generated group is at most exponential, this leaves a question whether the property A is
equivalent to a subexponential dimension growth.

The notion of decomposition complexity of a metric space was introduced in~\cite{GTY} (see also~\cite{NY}) in game theoretic terms.
It was shown that the finite decomposition complexity (FDC) implies property A. The finite decomposition complexity was verified for a large class
of groups, in particular for all countable subgroups of $GL(n,K)$ for an arbitrary field $K$~\cite{GTY}. In this paper we address the question what is the relation between asymptotic property C and FDC. It turns out that in the classical dimension theory there is no analog of FDC.
We plan to present one and compare it with Haver's property C in a future publication. Here to make a comparison of FDC and the asymptotic property C
we introduce the notion of the straight finite decomposition complexity sFDC opposed to the game theoretic finite decomposition complexity=gFDC=FDC.
We prove the following implications
$$
gFDC\ \Rightarrow\ sFDC\ \Rightarrow\ property\ A
$$
and
$$
\text{\em asymptotic property}\ C\ \Rightarrow sFDC.
$$
We know that the last implication is not an equivalence for metric spaces. We don't know if it is an equivalence
for groups.
Also we do not know if any of the first two implications is reversible.

In Section 5 of the paper we compare a game-theoretic approach with the standard in the classical dimension theory.
We did the comparison for the Haver's property C and found that a game-theoretic analog of it defines the countable dimensionality. It is known that these classes are different~\cite{E}.

We are thankful to Takamitsu Yamauchi for spotting a gap in the first version of our paper and supplying a reference to the work of L. Babinkostova.

\section{Preliminaries}

All spaces are assumed to be metrizable.

A generic metric is denoted by $d$. Given two nonempty subsets $A,B$ of $X$, we let $d(A,B)=\inf\{d(a,b)\mid a\in A,\ b\in B\}$.

Let $R>0$. We say that a family $\mathcal A$ of nonempty subsets of $X$ is $R$-disjoint if $d(A,B)>R$, for every $A,B\in\mathcal A$.

A metric space $X$ is geodesic if for every $x,y\in X$ there exists an isometric embedding $\alpha\colon [0,d(x,y)]\to X$ such that $\alpha(0)=x$ and $\alpha(d(x,y))=y$.

A metric space $X$ is discrete if there exists $C>0$ such that $d(x,y)\ge C$, for every $x,y\in X$, $x\neq y$.  A discrete metric space is said to be of bounded geometry if there
exists a function $f \colon\mathbb R_+ \to \mathbb R_+$ such that every ball of radius $r$ contains at most $f(r)$
points.

Let  $\mathcal X,\mathcal Y$ be  families of metric spaces and $R>0$. We say that $\mathcal X$ is $R$-decomposable over $\mathcal Y$ if, for any $X\in\mathcal X$, $X=\bigcup(\mathcal V_1\cup\mathcal V_2)$, where $\mathcal V_1,\mathcal V_2$ are $R$-disjoint families and $\mathcal V_1\cup\mathcal V_2\subset\mathcal Y$.

A family $\mathcal X$  of metric spaces is said to be {\em bounded} if $$\mathrm{mesh}(\mathcal X)=\sup\{\mathrm{diam}\, X\mid X\in\mathcal X\}<\infty.$$

 Let $\mathfrak A$ be a collection of metric families. A metric family $\mathcal X$ is {\em decomposable
over} $\mathfrak A$  if, for every $r > 0$, there exists a metric family $\mathcal Y \in\mathfrak A$ and an $r$-decomposition of
$\mathcal X$ over $\mathcal Y$.

\begin{defin} \cite{GTY}
We introduce the metric decomposition game of two players, a defender and a challenger. Let $\mathcal X = \mathcal Y_0$ be
the starting family. On the first turn the challenger asserts $R_1>0$, the defender responds by exhibiting an $R_1$-decomposition of $\mathcal Y_0$
over a new metric family $\mathcal Y_1$. On the second turn, the challenger asserts an integer $R_2$,
 the defender responds by exhibiting an $R_2$-decomposition of $\mathcal Y_1$ over a new metric family $\mathcal Y_2$. The game continues in this way, turn
after turn, and ends if and when the defender produces a bounded family. In this case the defender has won.

A metric family $\mathcal X$ has FDC if the defender has always the winning strategy. A metric space $X$ has FDC if  the family $\{X\}$ does.
\end{defin}
\begin{defin} We say that a metric space $X$ satisfies the {\em straight  Finite Decomposition Property} (sFDC) if, for any sequence $R_1<R_2<\dots$ of positive numbers, there exists $n\in\mathbb N$ and metric families such that $\mathcal V_1,\mathcal V_2,\dots, \mathcal V_n$ such that  $\{X\}$ is $R_1$-decomposable over $\mathcal V_1$,  $\mathcal V_i$ is $R_i$-decomposable over $\mathcal V_{i+1}$, $i=1,\dots, n-1$, and the family $\mathcal V_n$ is bounded.
\end{defin}

The following easy follows from the definition.
\begin{prop}\label{2}
Suppose that $X$ has the FDC. Then it has straight FDC.
\end{prop}

We recall that the {\em asymptotic dimension} of a metric space does not exceed $n$, $\as X\le n$ if for every $R>0$ there are
uniformly bounded $R$-disjoint families  $\mathcal U_i$, $i=0,\dots,n$ of sets in $X$ such that  the family $\cup_{i=1}^n\mathcal U_i$ is
a cover of $X$~\cite{Gr1}.

The following notion was introduced in \cite{Dr1}.
\begin{defin} A metric space $X$ is said to have the {\em asymptotic property C} if for every sequence $R_1<R_2<\dots$ there exists $n\in\mathbb N$ and uniformly bounded $R_i$-disjoint families $\mathcal U_i$, $i=1,\dots,n$, such that the family $\cup_{i=1}^n\mathcal U_i$ is a cover of $X$.
\end{defin}

We recall the definition of the property A~\cite{Yu2}.

\begin{defin} Let $X$ be a discrete metric space. We say that $X$ has Property A if for
every $\varepsilon > 0$ and $R > 0$ there exists a collection of finite subsets $\{A_x\}_{x\in X}$, $A_x \subset X \times\mathbb N$ and
a constant $S > 0$ such that
\begin{enumerate}
\item $\frac{\#(A_x \Delta A_y)}{\#(A_x \cap A_y)}\le\varepsilon$ when $d(x, y) < R$.
\item $ A_x \subset B(x, S ) \times \mathbb N$.
\end{enumerate}
\end{defin}

Let $X$ be a discrete metric space with
bounded geometry. It is known that $X$ has Property A if and only if for every $\varepsilon > 0$ and every $R > 0$ there
exists a probability measure valued map $x\mapsto \xi_x\colon X\to\ell_{1}(X)$   and a number $S > 0$ such that
\begin{enumerate}
\item  $\|\xi_x -\xi_y\|_1< \varepsilon$ if $d(x, y) < R$ and
\item $\mathrm{supp} (\xi_x )\subset B(x, S )$
\end{enumerate}
(see, e.g., \cite{HR}).

A map $f\colon X\to Y$ of metric spaces is bornologous (coarse uniform) if there exists a function $\varphi\colon [0,\infty)\to[0,\infty)$ such that $d(f(x),f(y))\le \varphi(d(x,y))$, for all $x,y\in X$. A map is metrically proper if the preimage of every bounded set is bounded. A map is coarse if it is both coarse uniform and metrically proper.

We say that two maps $f,g\colon X\to Y$ are close if $$d(f,g)=\sup\{d(f(x),g(x))\mid x\in X\}<\infty.$$

A map $f\colon X\to Y$ is a coarse equivalence if there exists a map $g\colon Y\to X$ such that the compositions $gf$ and $fg$ are close to the identity maps $1_X$ and $1_Y$ respectively.

These notions of coarse geometry have their counterparts for the families of metric spaces. A map $F\colon\mathcal X\to\mathcal Y$ of a metric family
$\mathcal X$ into a metric family
$\mathcal Y$ is a collection of maps $F=\{f\}$, where every $f$ maps some $X\in\mathcal X$ into some $Y\in \mathcal Y$.

We say that  $F\colon\mathcal X\to\mathcal Y$ is coarse uniform if there exists a nondecreasing function $\varrho\colon [0,\infty)\to[0,\infty)$ such that $d(f(x), f(y)) \le\varrho(d(x, y))$, for every $f\in F$ and every $x,y$ in the domain of $f$.

A family $F\colon\mathcal X\to\mathcal Y$ is metrically proper if there exists a nondecreasing function $\delta\colon [0,\infty)\to[0,\infty)$ such that $\delta(f(x), f(y)) \le d(x, y)$, for every $f\in F$ and every $x,y$ in the domain of $f$.

\section{Straight Finite decomposition complexity}

\begin{thm}[Coarse Invariance]\label{t:coarse} The property sFDC is  coarse invariant.
\end{thm}
\begin{proof} Let $f\colon X\to Y$ be a coarse equivalence and let $Y$ have the sFDC. Let $R_1<R_2<\dots$. By \cite[Lemma 3.1.1]{GTY}, there exists $S_1>0$ such that, whenever $Y=(\cup\mathcal Y_1)\sqcup  (\cup\mathcal Y_2)$ is an $S_1$-decomposition of $Y$, then $X=(\cup\mathcal X_1)\sqcup  (\cup\mathcal X_2)$, where $\mathcal X_i= \{f^{-1}(Z)\mid Y\in \mathcal Y_i\}$, $i=1,2$, is an $R_1$-decomposition of $Y$.

Next, we apply \cite[Lemma 3.1.1]{GTY} (which is actually formulated for metric families) to the maps $F_i\colon \mathcal X_i\to \mathcal Y_i$, $i=1,2$, and we obtain that there exists $S_2>0$ and an $S_2$-decomposition of every $Z\in \mathcal Y_1\cup\mathcal Y_2$ such that the preimages of the elements of this decomposition form an $R_2$-decomposition of $f^{-1}(Z)\in  \mathcal X_1\cup\mathcal X_2$.
We proceed similarly until the elements of the  $S_2$-decomposition of every $Z\in \mathcal Y_{n-1}\cup\mathcal Y_{n-1}$ form a  bounded metric family. Then their preimages also form a bounded metric family, by \cite[Lemma 3.1.2]{GTY}.

\end{proof}

\begin{prop}
Asymptotic property C implies straight FDC.
\end{prop}
\begin{proof}
Given  a sequence $R_1< R_2< R_3<\dots$, apply property C condition to get $\scr U_1,\dots,\scr U_m$, $R_i$-disjoint uniformly bounded families that
cover $X$. Then we define
the partition of $X$ into $\scr U_1$ and the one element  family that consists of the complement to $\cup\scr U_1$. Then consider the intersection of $\scr U_2$ with the complement, and so on.
\end{proof}

\begin{prop}\label{p:geo} Every discrete metric sFDC space can be isometrically embedded into a geodesic metric sFDC space.
\end{prop}
\begin{proof} We follow the construction from \cite{Dr1}. Let $X$ be a discrete metric sFDC space. For any $x,y\in X$, attach isometrically the metric segment $[0,d(x,y)]$ to $X$ along its endpoints. Endow the obtained space, $X'$, by the length metric. We assume that $X$ naturally lies in $X'$.

Let $R_1<R_2<\dots$ be a sequence of positive numbers. Let $$K_n=\{ x\in X'\mid n(R_1+1)\le d(x,X)<(n+1)(R_1+1)\}.$$ Note that every $K_n$ is coarsely equivalent to $X$.  Indeed, $K_n$ lies in the $(n+1)(R_1+1)$-neighborhood of $X$. On the other hand, since $X$ is unbounded, for every $x\in X$ there is $y\in K_n$ with $d(x,y)<(n+1)(R_1+1)$, i.e. $X$ lies in the $(n+1)(R_1+1)$-neighborhood of $K_n$. Therefore, $X$ and $K_n$ are of finite Hausdorff distance and therefore coarsely equivalent.

We represent $X'$ as $\bigcup(\mathcal V_1\cup \mathcal V_2)$, where
$${\mathcal V}_1=\{X\}\cup\{K_n\mid n=2i+1,\ i\in\mathbb N\},$$ $${\mathcal V}_2=\{K_n\mid n=2i,\ i\in\{0\}\cup\mathbb N\}.$$

Now the result follows from the fact that $X$ has sFDC and Proposition \ref{t:coarse}.
\end{proof}

A metric space $X$ is discrete if there exists $C>0$ such that $d(x,y)\ge C$, for every $x,y\in X$, $x\neq y$.  A discrete metric space is said to be of bounded geometry if there
exists a function $f \colon\mathbb R_+ \to \mathbb R_+$ such that every ball of radius $r$ contains at most $f(r)$
points. We recall the definition of the property A~\cite{Yu2}.

\begin{defin} Let $X$ be a discrete metric space. We say that $X$ has Property A if for
every $\varepsilon > 0$ and $R > 0$ there exists a collection of finite subsets $\{A_x\}_{x\in X}$, $A_x \subset X \times\mathbb N$ and
a constant $S > 0$ such that
\begin{enumerate}
\item $\frac{\#(A_x \Delta A_y)}{\#(A_x \cap A_y)}\le\varepsilon$ when $d(x, y) < R$.
\item $ A_x \subset B(x, S ) \times \mathbb N$.
\end{enumerate}
\end{defin}

An arbitrary metric space has Property A if it contains a coarsely dense discrete subset with Property A.

Let $X$ be a discrete metric space with
bounded geometry. It is known that $X$ has Property A if and only if for every $\varepsilon > 0$ and every $R > 0$ there
exists a probability measure valued map $x\mapsto \xi_x\colon X\to\ell_{1}(X)$   and a number $S > 0$ such that
\begin{enumerate}
\item  $\|\xi_x -\xi_y\|_1< \varepsilon$ if $d(x, y) < R$ and
\item $\mathrm{supp} (\xi_x )\subset B(x, S )$
\end{enumerate}
(see, e.g., \cite{HR}).

\begin{thm}\label{3}
Straight FDC implies property A for  metric spaces.
\end{thm}
\begin{proof}

We need some preliminary facts and assertions.

{\bf Open 2-covers.} An open cover of a metric space $(X,d)$ by two sets  $\scr U=\{U_+,U_-\}$
will be called a {\em 2-cover} of $X$. Suppose that  $X=\coprod V_i$.

For any 2-cover of $X$ we define a map $f\colon X\to P(\scr U)\subset\Delta\subset\ell_1(\scr U)$ to the space of probability measures supported on the 2-point set $\scr U$ by the formula $f(x)=\alpha_{U_+}(x)\delta_{U_+}+\alpha_{U_-}(x)\delta_{U_-}\in P(\mathcal U)$ where for $x\in V_i$,
$$
\alpha_{U_+}(x)= \frac{d(x,V_i\setminus U_+)}{S_x}\ \ \text{and}\ \ \alpha_{U_-}=\frac{d(x,V_i\setminus U_-)}{S_x}$$
where $S_x=d(x,V_i\setminus U_+)+d(x,V_i\setminus U_-))$.
\begin{ass}\label{2-cover}
Let $\scr U=\{U_+,U_-\}$ be a 2-cover of $X=\cup\mathcal V$, where $\mathcal V$ is a disjoint locally finite family.
Let $\lambda$ be a Lebesgue number of the restriction of the 2-cover $\scr U|_{V}$ to $V$ for all $V\in\mathcal V$. Then the map $f\colon X\to P(\scr U)$ is locally  $(6/\lambda)$-Lipschitz.
\end{ass}
\begin{proof}
Note that for $x,y\in X$ we obtain $$\|f(x)-f(y)\|_{\ell_1}=|\alpha_{U_+}(x)-\alpha_{U_-}(y) |+
|\alpha_{U_-}(x)-\alpha_{U_-}(y)|.$$
For $x,y\in V=V_i$ the triangle inequality for $d$ implies that
$$
|d(x,V\setminus U_+)-d(y,V\setminus U_+)|\le d(x,y).$$
This and the fact that $S_x\ge\lambda$ imply for the first summand
$$
\left |\frac{d(x,V\setminus U_+)}{S_x}-\frac{d(y,V\setminus U_+)}{S_y}\right |\le\frac{d(x,y)}{S_x}+d(y,V\setminus U_+)\left |\frac{1}{S_x}-\frac{1}{S_y}\right |$$
$$
=\frac{d(x,y)}{S_x}+\frac{d(y,X\setminus U_+)}{S_y}\frac{|S_x-S_y|}{S_x}\le\frac{3d(x,y)}{S_x}\le\frac{3d(x,y)}{\lambda}.$$
Similarly for the second summand
$$
\left |\frac{d(x,V\setminus U_-)}{S_x}-\frac{d(y,V\setminus U_-)}{S_y}\right |\le \frac{3d(x,y)}{\lambda}.$$
\end{proof}

{\bf Binary covers.} A {\em binary refinement} $\scr V$ of an open cover $\scr U$ is a cover which is obtained
by replacing every set $U\in\scr U$ by  two sets $U_+,U_-$ that form an open 2-cover of $U$.

An open cover $\scr U^m$ of a metric space $X$ is called {\em a binary cover} if it is constructed recursively
by taking $m$ binary refinements starting  from an open 2-cover $\scr U^1$. Note that $\scr U^m$
admits a tower of refinements
$$
\scr U^1>\scr U^2>\cdots>\scr U^m
$$
with the 2-to-1  refinement maps $\phi^{i+1}_i:\scr U^{i+1}\to\scr U^i$, $W\subset\phi^{i+1}_i(W)$.
Therefore, the cardinality of $\scr U^m$ is $2^m$. We use the notation $U_{\pm}$ for the preimage
$(\phi^{i+1}_i)^{-1}(U)$ of $U$. Thus, $\{U_+,U_-\}$ is an open 2-cover of $U$.

Additionally we assume that for each $U\in\scr U^i$, $i=1,\dots, m$, the sets $U_\pm$ admit disjoint decompositions $U_+=\coprod V^U$ and $U_-=\coprod W^U$.

For a binary cover $\scr U^m$ we define recursively a sequence of maps $f_i:X\to P(\scr U^i)$, $i=1,\dots, m$, to the  probability measures
on $\scr U^i$ as follows: We set $f_1(x)=\alpha_{U_+}(x)\delta_{U_+}+\alpha_{U_-}(x)\delta_{U_-} $ for the 2-cover $\scr U^1=\{U_+,U_-\}$ with $V=X$.  Then we define
$$\mu_U:U\to P((\phi^{2}_1)^{-1}(U))\subset P(\scr U^2)$$ to be the map $f$ for the 2-cover $(\phi^{i+1}_i)^{-1}(U)=\{U_+,U_-\}$ of $U$ for each of the two sets $U\in\scr U^1$.
Then we define
$$f_2(x)=\sum_{x\in U\in\scr U^1}\alpha_U(x)\mu_U(x).$$

Generally, if $f_i(x)=\sum_{U\in\scr U^i}\beta_U(x)\delta_U$, then
$$
f_{i+1}(x)=\sum_{x\in U\in\scr U^i}\beta_U(x)\mu_U(x)
$$
where $\mu_U:U\to P((\phi^{i+1}_i)^{-1}(U))\subset P(\scr U^{i+1})$ is the map $f$ defined for the 2-cover $(\phi^{i+1}_i)^{-1}(U)=\{U_+,U_-\}$ of $U$ for each of the $2^i$ sets $U\in\scr U^i$. Thus, $\mu_U(x)=\alpha_{U_+}(x)\delta_{U_+}+\alpha_{U_-}(x)\delta_{U_-}$. We may assume that
$\beta_U$ is defined on all of $X$ with $\beta_U(x)=0$ for $x\not\in U$. Thus, $
f_{i+1}(x)=\sum_{U\in\scr U^i}\beta_U(x)\mu_U(x)
$.

Note that $f_i=P(\phi^{i+1}_i)\circ f_{i+1}$.
\begin{ass}\label{binary cover}
Let $X$ be a geodesic metric space.
Let $\lambda_i$ be a Lebesgue number of the cover $\scr U^i$ restricted to each $V^U$ and each $W^U$ for all $U\in\scr U^{i-1}$, $i=1,\dots,m$. Then the map $f_m$ is $(3\sum_{i=1}^m\frac{2^i}{\lambda_i})$-Lipschitz.
\end{ass}
\begin{proof} Since $X$ is geodesic, it suffices to prove that  the map $f_m$ is locally $(3\sum_{i=1}^m\frac{2^{i}}{\lambda_i})$-Lipschitz. Let $d(x,y)\le 1$.

Induction on $m$. Assertion~\ref{2-cover} is the base of induction with $\scr V=\{X\}$.

Note that $$\|f_m(x)-f_m(y)\|_{\ell_1}=\|\sum_{x,y\in U\in\scr U^{m-1}}\beta_U(x)\mu_U(x)-\beta_U(y)\mu_U(y)\|_{\ell_1}=$$
$$\sum_{x,y\in U\in\scr U^{m-1}}\|\beta_U(x)\mu_U(x)-\beta_U(y)\mu_U(y)\|_{\ell_1}+$$

$$\sum_{y\not\in U, x\in U\in\scr U^{m-1}}\|\beta_U(x)\mu_U(x)-\beta_U(y)\mu_U(y)\|_{\ell_1}+$$
$$
\sum_{x\not\in U, y\in U\in\scr U^{m-1}}
\|\beta_U(x)\mu_U(x)-\beta_U(y)\mu_U(y)\|_{\ell_1}.$$

For each summand  the triangle inequality implies that
$$
\|\beta_U(x)\mu_U(x)-\beta_U(y)\mu_U(y)\|_{\ell_1}\le\|\beta_U(x)\mu_U(x)-\beta_U(x)\mu_U(y)\|_{\ell_1}+$$
$$|\beta_U(x)-\beta_U(y)|\|\mu_U(y)\|_{\ell_1}
\le\beta_U(x)\|\mu_U(x)-\mu_U(y)\|_{\ell_1}+|\beta_U(x)-\beta_U(y)|.$$
Similarly,
$$
\|\beta_U(x)\mu_U(x)-\beta_U(y)\mu_U(y)\|_{\ell_1}\le
\beta_U(y)\|\mu_U(x)-\mu_U(y)\|_{\ell_1}+|\beta_U(x)-\beta_U(y)|.$$
Thus, if $x\not\in U$ or $y\not\in U$ we have
$$
\|\beta_U(x)\mu_U(x)-\beta_U(y)\mu_U(y)\|_{\ell_1}\le |\beta_U(x)-\beta_U(y)|.$$

Note that by Assertion~\ref{2-cover} $\|\mu_U(x)-\mu_U(y)\|_{\ell_1}\le\frac{6d(x,y)}{\lambda_m}$  for sufficiently close $x$ and $y$.

Thus,
$$\|f_m(x)-f_m(y)\|_{\ell_1}\le\sum_{x,y\in U\in\scr U^{m-1}}\left (\beta_U(x)\frac{6d(x,y)}{\lambda_m}+|\beta_U(x)-\beta_U(y)|\right )+ $$
$$\sum_{y\not\in U, x\in U\in\scr U^{m-1}}|\beta_U(x)-\beta_U(y)|+\sum_{x\not\in U, y\in U\in\scr U^{m-1}}|\beta_U(x)-\beta_U(y)|\le $$
$$\sum_{U\in\scr U^{m-1}}\left(\beta_U(x)\frac{6d(x,y)}{\lambda_m}+|\beta_U(x)-\beta_U(y)|\right)$$
$$\le 6\frac{2^{m-1}d(x,y)}{\lambda_m}+\|f_{m-1}(x)-f_{m-1}(y)\|_{\ell_1}\le 3d(x,y)\left(\sum_{i=1}^{m-1}\frac{2^i}{\lambda_i}+\frac{2^{m}}{\lambda_m}\right).$$
The last inequality is by the induction assumption.
\end{proof}

Given $n\in\N$, we apply the definition of sFDC to $X$ with the sequence $R_i=4^in$
to obtain a nested sequence of partitions $\scr F_1>\scr F_2>\dots>\scr F_m$ that ends with a uniformly bounded partition. Recall that the partition
$\scr F_1$ splits in two $R_1$-disjoint families $\scr F_1=\scr F_1^+\cup \scr F_1^-$, and, generally,
each tile $F\in\scr F_i$ is decomposed into two $R_{i+1}$-disjoint families  $\scr F_{i+1}|_F=\scr F_{i+1}|_F^+\cup \scr F_{i+1}|_F^-$.

We construct a sequence of open
covers $\mathcal V_i$, $i=1,\dots, m$, of $X$ by enlarging the tiles of the partitions $\scr F_i$. For tiles $F$ in $\scr F_1$ we take
the open $R_1/3$-neighborhoods $V_F=N_{R_1/3}(F)$.
For each tile $F\in\scr F_2$ we take $V_F=N_{R_2/3}(F)\cap N_{R_1/3}(F')$ where $F'\in\scr F_1$ is the unique
tile such that $F$ is defined by the partitioning of $F'$ and so on. Thus for each $V\in\scr F_{i+1}$ there is a uniquely determined $V'\in\scr V_i$
such that $V\subset V'$.

Next we note that each cover $\scr V_i$ is a refinement of a binary cover $\scr U^i$ where $\scr U^i$ is defined recursively as follows.
The 2-cover $\scr U^1$ consists of $$U^+_1=\cup_{F\in\scr F_1^+} V_F\ \ \text{and}\ \ U^-_1=\cup_{F\in\scr F_1^-} V_F.$$
The cover $\scr U^2$ consists of four sets
$$U=\cup_{F'\in\scr F_1^\pm, F\in\scr F_2|_{F'}^\pm} V_F$$ an so on. We note that every element $U\in\scr U^i$ is the disjoint union of open sets from $\scr V_i$ called {\em components of} $U$.
In fact the components of every $U\in\scr U^i$ are $R_i/3$-disjoint. There is a well-defined component map $c_i:\scr V_i\to\scr U^i$ where $V$ is a component of $c_i(V)$.

Also note that $\lambda_i=R_i/3$ is a Lebesgue number of $\scr U^i$ restricted to each component.
\begin{ass}\label{lift}
Suppose that $X$ is a geodesic metric space. Then the map $f_m:X\to P(\scr U^m)$ admits a $(18/n)$-Lipschitz lift $g:X\to P(\scr V_m)$ with respect to the component map on the probability measures
$P(c_m):P(\scr V_m)\to P(\scr U^m)$.
\end{ass}
\begin{proof}
Note that for every $x\in U\in\scr U^m$ there is a unique component $V=V(U,x)\in\scr V_m$ that contains $x$.
This defines for each $x\in X$  a section $$s_x:supp(f(x))\to\scr V_m$$ of $c_m$ over the support of of $x$.
Then we set $g(x)=P(s_x)(f_m(x))$.

Since the components of $U\in\scr U^m$ are $(R_m/3)$-disjoint, for $x,y\in X$ with $d(x,y)\le 1$ we have $s_x\equiv s_y$ on the intersection $supp(f(x)\cap supp(f(y))$.
This implies that $$\|g(x)-g(y)\|_{\ell_1}=\|f(x)-f(y)\|_{\ell_1}.$$
By Assertion~\ref{binary cover},
$$\|f_m(x)-f_m(y)\|_{\ell_1}\le 3\sum_{i=1}^m\frac{2^{i}}{\lambda_i}\le 12\sum_{i=1}^m\frac{2^{i}}{4^in}\le\frac{12}{n}.
$$
Since $X$ is geodesic, it follows that $g$ is $(12/n)$-Lipschitz.
\end{proof}

For each $V\in\scr V_m$ we fix a point $y_v\in V$. This defines a map $h: P(\scr V_m)\to P(X)$.
For  every $x\in X$ we define a probability measure
$a^n(x)=hg(x)$.

Now we assume that $X$ is geodesic. Then  by Assertion~\ref{lift} $$\|a^n(x)-a^n(y)\|_{\ell_1}\le 12K/n$$
for all $n$, if $d(x,y)\le K$. Thus, the condition (1) of the property A is satisfied.
The measures  $a^n(x)$ have uniformly bounded supports since the cover $\scr V_m$ is uniformly bounded.
This takes care of the condition (2).

If $X$ is not geodesic, we assume, without loss of generality, that $X$ is discrete. Then we apply Proposition \ref{p:geo} and isometrically embed $X$ into a geodesic metric space $X'$. Since, by what is proved above, $X'$ has property A, the space $X$ has also property A, by \cite{Ro}.
\end{proof}

\begin{thm}[Finite Sum Theorem] Let $Z$  be a metric space such that $Z=X\cup Y$, where $X$ and $Y$ satisfy sFDC. Then $Z$ also satisfies sFDC.
\end{thm}
\begin{proof} Let $R_1<R_2<R_3<\dots$. We consider an $R_1$-decomposition $Z=(\cup\mathcal Y_1)\cup(\cup\mathcal Y_2)$, where $\mathcal Y_1=\{X\}\cup \{\{z\in Z\mid (2k-1) (R_1+1)<d(z,X)\le 2k(R_1+1)\}\mid k\in\mathbb N\}$,
$\mathcal Y_2=\{\{z\in Z\mid 2k (R_1+1)<d(z,X)\le (2k+1)(R_1+1)\}\mid k\in\mathbb N\}$. Then the assertion follows from the fact that $X$ and every subspace of $Y$ satisfy sFDC.
\end{proof}

By a slight modification, one can define the notion of sFDC for the metric families. We say that a metric family $\mathcal X$ satisfies the sFDC, if, for any $R_1<R_2<\dots$, there exist $n\in\mathbb N$ and  metric families $\mathcal X_i$, $i=1,\dots,n$, such that $\mathcal X_1=\mathcal X$, $\mathcal X_n$ is a bounded family and $\mathcal X_{i+1}$ is $R_i$-decomposable over $\mathcal X_i$ for all $i=1,\dots, n-1$. Clearly, a family $\{X\}$ has the sFDC if and only if so does the space $X$. Also, any subfamily of a family that has the sFDC, has the sFDC as well.

\begin{thm}[Sum Theorem] Let $X$  be a metric space such that $X=\cup\mathcal X$, where $\mathcal X$ has sFDC   and the following
condition holds: for any $r>0$, there exists $Y(r)\subset X$ such that $Y(r)$ has sFDC and the family $\mathcal X(r)=\{X\setminus Y(r)\mid X\in \mathcal X\}$ is $r$-disjoint. Then $X$ also satisfies sFDC.
\end{thm}
\begin{proof} Given $R_1<R_2<\dots$, find $Y(R_1)\subset X$ such that $Y(R_1)$ has sFDC and the family $\mathcal X(R_1)$ is $R_1$-disjoint. We see that $X$ is $R_1$-decomposable over the family $\{Y(R_1)\}\cup \mathcal X(R_1)$ and $X=\cup\{Y(R_1)\}\cup(\cup\mathcal X(R_1))$.

Since the space $Y(R_1)$ and the family $\mathcal X$ have the sFDC, we conclude that there exist a natural number $n$ and families $\mathcal X_i$, $i=2,\dots,n$, such that $\{Y(R_1)\}\cup \mathcal X(R_1)$ is $R_2$-decomposable over $\mathcal X_2$, $\mathcal X_i$ is $R_{i+1}$-decomposable over $\mathcal X_{i+1}$ for all $i=2,\dots,n-1$, and $\mathcal X_n$ is a bounded family.
\end{proof}

\begin{ex}
In view of Proposition~\ref{2} the examples of groups with sFDC come from examples of groups with FDC like $\oplus_{i=1}^{\infty}\Z$ with the metric
$$d((x_i),(y_i))=\sum_{i=1}^\infty i|x_i-y_i|$$ (see~\cite{NY}). The natural groups to investigate here are $\Z\wr \Z$
with the word metric and
$\oplus_{i=1}^{\infty}\Z$ with the metric $$d((x_i),(y_i))=\sum_{i\in I}|x_i-y_i|+i$$ where $I=\{i\in\N\mid x_i\ne y_i\}$.

In view of Theorem~\ref{3} the examples of groups without sFDC come from groups without property A. Thus, Gromov monster groups are such
~\cite{Gr2},\cite{AD}. They are finitely presented and even could have a finite c Eilenberg-McLane complex $K(\pi,1)$~\cite{Sa}. Infinitely generated groups without property A and hence without sFDC are easier to produce. For example,  the Thompson group $F$ with the word metric with respect to the  generating set
$$F=\langle x_0,x_1,x_2,\dots\mid x_nx_k=x_kx_{n+1}\text{ for all }k<n\rangle$$ does not have  property A~\cite{WCh}. Of course, as it has been mentioned in the introduction, the question whether the Thompson group $F$ has the property A for  a proper invariant metric is an open problem.
\end{ex}

\section{Game theoretic asymptotic property C}

In the spirit of the FDC, one can define a game theoretic version of property C.

\begin{defin} We say that a metric space $X$ has the game theoretic asymptotic property C if there is a winning strategy for player I in the following game. Player II challenges player I by choosing $R_1>0$, then player I chooses an $R_1$-disjoint uniformly bounded family $\mathcal U_1$. Then player II  chooses $R_2>0$ and player I chooses an $R_1$-disjoint uniformly bounded family $\mathcal U_2$ and so on. Player I wins if there is $k$ such that the family $\cup_{i=1}^k\mathcal U_i$ is a cover of $X$.
\end{defin}

Let $\mathcal V$ be a family of nonempty subsets of a metric space $X$ and $R>0$. We say that $V,W\in\mathcal V$ are $R$-connected if there exist $V=V_0,V_1,\dots,V_n=W$ in $\mathcal V$ and $y_0\in V_0$, $x_i,y_i\in V_i$, $i=1,\dots,n-1$, and $x_n\in V_n$ such that $d(y_i,x_{i+1})\le R$ for every $i=0,\dots,n-1$. We denote by $\mathcal V^R$ the family whose elements are the unions of the equivalence classes of the $R$-connectedness relation.

\begin{prop}
A space $X$ has the game theoretic asymptotic property C if and only if $\mathrm{asdim}\, X=0$.
\end{prop}

\begin{proof} Suppose that $\mathrm{asdim}\, X=0$. Then given $R>0$, player I is able to find a uniformly bounded $R$-disjoint cover of $X$.

Now, suppose that  $\mathrm{asdim}\, X>0$. There is $R>0$ for which there is no $R$-disjoint uniformly bounded cover of $X$. Let $R_1=R+1$. Suppose that player I has made $n$ moves and chooses $R_i$-disjoint uniformly bounded families $\mathcal U_i$, $i=1,\dots,n$. Then player II takes $R_{n+1}=R_n+\mathrm{mesh}(\mathcal U_n)$.

Suppose that there is $k$ such that $\mathcal U=\cup_{i=1}^k\mathcal U_i$ is a cover of $X$.

For every $A\in\mathcal U_k$, let $A'=\cup\{B\in \mathcal U_{k-1}\mid d(A,B)<R_k/4 \}$.
Define $$\mathcal V_{k-1}=\{A'\mid A\in\mathcal U_k\}\cup \{B\in \mathcal U_{k-1}\mid B\cap A'=\emptyset\text{ for all }A\in\mathcal U_k\}.$$ From the choice of $R_k$ it follows that the family $\mathcal V_{k-1}$ is $R_{k-1}$-disjoint,  uniformly bounded and $\cup\mathcal V_{k-1}=\cup(\mathcal U_k\cup\mathcal U_{k-1})$.

Proceeding as above one can construct families $\mathcal V_i$, $i=k-1, k-2,\dots,2,1$, with the following properties:
\begin{enumerate}
\item $\mathcal V_i$ is uniformly bounded;
\item $\mathcal V_i$ is $R_i$-disjoint;
\item $\cup\mathcal V_i=\cup(\mathcal U_k\cup\mathcal U_{k-1}\cup\dots\cup \mathcal U_i)$.
\end{enumerate}

Therefore, the family  $(\cup_{i=1}^k\mathcal U_i)^{R_1}$ is an $R$-disjoint uniformly bounded cover of $X$ and we obtain a contradiction.
\end{proof}

Let $G$   denote the group $\bigoplus_{i=1}^{\infty}\Z$ supplied with the proper  metric
$$d((x_i),(y_i))=\sum_{i=1}^\infty i|x_i-y_i|.$$
It was proven in~\cite{NY}, Proposition 2.9.1 that $G$  has FDC.
\begin{question}\label{Q}
(a) Does the group $G$
have  asymptotic property C?

(b) Does every group with sFDC have asymptotic property C?
\end{question}
It turns out that the answer to the part (a) of the question is affirmative~\cite{Ya}.

\section{Game theoretic approach in classical dimension theory}

We recall that a space $X$ is called countable dimensional if it can be presented as the countable union of 0-dimensional subsets.

\begin{lem}\label{c-d}
Suppose that for a compactum $X$ for every $\epsilon>0$ there is a disjoint family of open sets $U_1,\dots U_k$
of $diameter(U_i)<\epsilon$ with $X\setminus W_{\epsilon}$  countable dimensional where
$W_{\epsilon}=\cup_{i=1}^kU_k$. Then $X$ is countable dimensional.
\end{lem}
\begin{proof} Let $X_n=X\setminus W_{1/n}$.
Note that $$X=\cup_{n=1}^{\infty}W_{\frac{1}{n}}\cup(\cap_{n=1}^{\infty}X_n).$$
We show that $F=\cap_{n=1}^{\infty}X_n$ is 0-dimensional. For that we show that $F$ is homeomorphic
to a subset of 0-dimensional compactum obtained as the inverse limit of a sequence
 $$
Y_1 \stackrel{\phi^2_1}\leftarrow Y_2\stackrel{\phi^3_2}\leftarrow Y_3\leftarrow\dots.
$$
Here $Y_1$ is the disjoint union $\coprod\overline{U^1_i}$ of the closures in $X$ of elements of the $\epsilon$-family with $\epsilon=1$,
$Y_2$ consists of the disjoint union of the closures of intersections $U^1_i\cap U^2_j$ with the bonding map $\phi^2_1:Y_2\to Y_1$
the union of the inclusions and so on.
\end{proof}

The following notion is introduced by Haver \cite{H}.

A
 metric space $(X, d)$ is said to have {\em property C} if for
each sequence of positive numbers $\{\varepsilon_i\}_{i=1}^\infty$, there exists a sequence of disjoint collections of
open sets $\{\mathcal U^i\}_{i=1}^\infty$ such that $\mathrm{mesh}\,\mathcal U^i<\varepsilon_i$, $i\in\mathbb N$, and $\cup_{i=1}^\infty\mathcal U^i$ is a cover of $X$.

\begin{defin}
We say that a metric space $X$ has the {\em game theoretic property C} if there is a winning strategy for player I in the following game. Player II challenges player I by choosing $\varepsilon_1>0$, then player I chooses a disjoint family of open sets $\mathcal U_1$ with $\mathrm{mesh}\ \mathcal U_1<\varepsilon_1$. Then player II  chooses $\varepsilon_1>0$ and player I chooses a disjoint family of open sets $\mathcal U_2$ with $\mathrm{mesh}\ \mathcal U_2<\varepsilon_2$.and so on. Player I wins if  the family $\cup_{i=1}^k\mathcal U_i$ is a cover of $X$.
\end{defin}

\begin{thm}
A compact metric space $X$ has the game theoretic property C if and only if it is countable dimensional.
\end{thm}
\begin{proof}
First, we show that every compact countable dimensional space $X$ has the game theoretic property C. Let $X=\cup_{i=1}^\infty Y_i$, where all $Y_i$ are zero-dimensional.

Let $\varepsilon_1>0$. Find a disjoint cover $\mathcal V_1$ with $\mathrm{mesh}\ \mathcal V_1<\varepsilon_1/2$ of $Y_1$ by clopen (in $Y_1$) subsets. Remark that, for every $V\in \mathcal V_1$ and every $x\in V$, we have $d(x, (\cup\mathcal V_1)\setminus V)>0$.   We let $$U_V=\cup\{B_{\min\{\varepsilon_1/4,d(x, (\cup\mathcal V_1)\setminus V)\}}(x)\mid x\in V\}.$$ The family $\mathcal U_1=\{U_V\mid V\in\mathcal V_1\}$ is a disjoint family of open in $X$ sets such that  $\mathrm{mesh}\ \mathcal U_1<\varepsilon_1$ and $\cup\mathcal U_1\supset Y_1$.

Given $\varepsilon_i>0$, one can similarly find a disjoint family $ \mathcal U_i$ of open in $X$ sets such that  $\mathrm{mesh}\ \mathcal U_i<\varepsilon_i$ and $\cup\mathcal U_i\supset Y_i$. Then the family $\cup_{i=1}^\infty\mathcal U_i$ is a cover of  $\cup_{i=1}^\infty Y_i=X$. Since $X$ is compact, there exists $k\in\mathbb N$ such that $\cup_{i=1}^k\mathcal U_i$ is a cover of $X$.

Now, assume that $X$ is not countable dimensional. By Lemma~\ref{c-d} there is $\epsilon_1$ such that for every open disjoint $\epsilon_1$ family
 $U_1,\dots U_k$ the complement is not countable dimensional. Then the the first player (Bad) pick up this $\epsilon_1$. No matter
 what first player (Good) does the complement will not be countable dimensional. Therefore by Lemma~\ref{c-d} there is number $\epsilon_2$
and so on. The process will never stop, i.e, the Good player will never win.
\end{proof}

REMARK. This theorem was also proven in~\cite{Bab}.

\end{document}